\documentclass[12pt]{article}
\usepackage{bbm}
\usepackage{psfrag,epic,eepic,epsfig}
\usepackage{fullpage,color}
\usepackage[applemac]{inputenc} 
\usepackage{amsmath,amsfonts,amssymb,amsthm,mathrsfs}
\usepackage[a4paper,vmargin={3.5cm,3.5cm},hmargin={2.5cm,2.5cm}]{geometry}
\usepackage[font=sf, labelfont={sf,bf}, margin=1cm]{caption}
\usepackage{graphicx,graphics}
\usepackage{latexsym}
\usepackage{ae,aecompl}
\usepackage[english]{babel}
 \usepackage[colorlinks=true]{hyperref}
\usepackage{enumerate}
\usepackage[normalem]{ulem}

\newcommand{\E}[1]{\ensuremath{\mathbb{E} \left[#1 \right]}}
\newcommand{\Prob}[1]{\ensuremath{\mathbb{P} \left(#1 \right)}}

\newcommand{\R}{\ensuremath{\mathbb{R}}}

\newcommand{\N}{\ensuremath{\mathbb{N}}}

\newcommand{\equidist}{\ensuremath{\stackrel{\mathrm{(d)}}{=}}}

\renewcommand{\d}{\ensuremath{\mathrm{d}}}

\newtheorem{thm}{Theorem}[section]

\newtheorem{lem}[thm]{Lemma}
\newtheorem{prop}[thm]{Proposition}
\newtheorem{cor}[thm]{Corollary}

\newtheorem{rem}[thm]{Remark}

\date{}

\title{\bf{A line-breaking construction of the stable trees}}

\author{Christina Goldschmidt\thanks{ Department of Statistics and Lady Margaret Hall, University of Oxford, E-mail: goldschm@stats.ox.ac.uk} \ \& B\'en\'edicte Haas\thanks{ Universit\'e Paris-Dauphine and \'Ecole normale sup\'erieure,  E-mail: haas@ceremade.dauphine.fr}}

\begin{document}

\maketitle

\abstract{We give a new, simple construction of the $\alpha$-stable tree for $\alpha \in (1,2]$. We obtain it as the closure of an increasing sequence of $\R$-trees inductively built by gluing together line-segments one by one.  The lengths of these line-segments are related to the the increments of an increasing $\R_+$-valued Markov chain.  For $\alpha = 2$, we recover Aldous' line-breaking construction of the Brownian continuum random tree based on an inhomogeneous Poisson process.}

\medskip


\medskip


\setlength{\parindent}{0pt}
\setlength{\parskip}{6pt}

\section*{Introduction}

The stable trees were introduced by Duquesne and Le Gall~\cite{DuquesneLeGall, DLG05}, building on earlier work of Le Gall and Le Jan~\cite{LeGallLeJan}.  These trees are intimately related to continuous state branching processes, fragmentation and coalescence processes, and appear as scaling limits of various models of trees and graphs.
More precisely, they form a family of random compact $\R$-trees $$(\mathcal T_{\alpha},\ \alpha \in (1,2])$$ which represent the genealogies of the continuous-state branching processes with branching mechanism $\lambda \mapsto c \lambda^{\alpha}$ for $\alpha \in (1,2]$ (where the choice of the constant $c$ fixes the normalization of the tree of index $\alpha$).  As such, they constitute all of the possible scaling limits of Galton--Watson trees whose offspring distributions have mean 1 and lie in the domain of attraction of a stable law of parameter $\alpha \in (1,2]$, conditioned to have a fixed total progeny $n$.  In particular, by well-known theorems of Aldous \cite{AldousCRT3} and Duquesne \cite{Du03}, if the offspring distribution $(p_k)_{k \ge 0}$ of the Galton--Watson tree has mean 1 and variance $\sigma^2 \in (0,\infty)$ then, writing $T^{\mathrm{GW}}_n$ for the conditioned tree, we obtain
\begin{equation}
\label{cvGW1}
\frac{T^{\mathrm{GW}}_n}{\sqrt{n}} \overset{\mathrm{(d)}}\longrightarrow \frac{1}{\sigma\sqrt{2}} \cdot \mathcal{T}_2,
\end{equation}
as $n \to \infty$, in the sense of the Gromov-Hausdorff distance.  The tree $\mathcal{T}_2$ is $\sqrt{2}$ times the Brownian continuum random tree (CRT) of Aldous~\cite{AldousCRT1,AldousCRT2,AldousCRT3}, which has also been shown to be the scaling limit of various other natural classes of trees \cite{HaasMiermont,MarckertMiermont,MiermontInvariance}, graphs and maps \cite{AlMar08, Bet14,  Car14, CurienHaasKortchemski,JanSte14}. If, on the other hand, $p_k \sim C k^{-1-\alpha}$ as $k \to \infty$ for $\alpha \in (1,2)$, then 
\begin{equation}
\label{cvGW2}
\frac{T^{\mathrm{GW}}_n}{n^{1 - 1/\alpha}} \overset{\mathrm{(d)}}\longrightarrow\left(\frac{\alpha-1}{C \Gamma(2-\alpha)} \right)^{1/\alpha} \alpha^{1/\alpha -1} \cdot \mathcal{T}_\alpha,
\end{equation}
as $n \to \infty$, again in the Gromov-Hausdorff sense. We have chosen a somewhat unusual normalization of our trees in order to streamline the presentation of our results; for $\alpha \in (1,2]$, our $\mathcal{T}_\alpha$ is $\alpha$ times the standard $\alpha$-stable tree of Duquesne and Le Gall (which had for branching mechanism $\lambda \mapsto \lambda^{\alpha}$). 
Let us briefly recall the formalism used  here:  an $\R$-tree is a metric space $(\mathcal T,d)$  such that for every pair of points $
x,y \in \mathcal T$,  there exists an isometry $\varphi_{x,y}:[0,d(x,y)]\to \mathcal T$ such that $\varphi_{x,y}(0)=x$, $\varphi_{x,y}(d(x,y))=y$, and the image of this isometry is the unique continuous injective path from $x$ to $y$.
In this paper, we  always consider rooted $\R$-trees (by simply distinguishing a vertex) and we identify two  $\mathbb R$-trees when there exists a root-preserving isometry between them. Abusing notation slightly, we will also use $(\mathcal T,d)$ to represent an isometry class. The set of compact rooted (isometry classes of) $\mathbb R$-trees is then endowed with the Gromov-Hausdorff distance, which makes it Polish. We refer to \cite{BBI01,Eva08,LG06} for the details and for more background on this topic. We will often use the notation $\mathcal T$ instead of $(\mathcal T,d)$, the distance being implicit. We will then write $a\cdot \mathcal T$ for the tree $(\mathcal T, ad)$.

In the last few years, the geometric and fractal aspects of stable trees have been studied in great detail: Hausdorff and packing dimensions and measures  \cite{DLG05, DLG06, Du12, HM04}; spectral dimension \cite{CrH10};  spinal decompositions and invariance under uniform re-rooting \cite{HaasPitmanWinkel,DLG09}; fragmentation into subtrees \cite{Mier03, Mier05}; and embeddings of stable trees into each other \cite{CurienHaas}. The stable trees are also related to Beta-coalescents \cite{AD13,BBS07}; intervene in the description of other scaling limits of random maps \cite{LGBrMap13,MierBrMap13, LGM11}; and have dual graphs, called the stable looptrees \cite{CRLoop13}, which also appear as scaling limits of natural combinatorial  models. 

We now proceed to introduce our new construction.  In order to do so, we recall that the stable tree $\mathcal T_{\alpha}$ is naturally endowed with a uniform probability measure $\mu_{\alpha}$, which is supported on its set of leaves \cite{DLG05}. This measure can be obtained by considering the empirical measure on the vertices of $T_n^{\mathrm{GW}}$ in (\ref{cvGW1}) and (\ref{cvGW2}): we then have the joint convergence of these measured rescaled conditioned Galton-Watson trees  towards the measured tree $(\mathcal T_{\alpha},\mu_{\alpha})$. 

In general, the tree $\mathcal T_{\alpha}$ may be viewed in several equivalent ways.  We may also think of it as the $\R$-tree coded by an excursion of the so-called \emph{stable height process} which, in the case of the Brownian CRT, is simply a (constant multiple of a) standard Brownian excursion. For $\alpha \in (1,2)$, the stable height process is a rather complicated random process; see \cite{DuquesneLeGall} for this perspective, which we will not consider further here.  Another way to define  $\mathcal T_{\alpha}$ is via its  \emph{random finite-dimensional distributions}, which describe the joint distribution as $p$ varies of the subtrees of $\mathcal{T}_{\alpha}$ spanned by the root and $p$ uniformly-chosen leaves.  More precisely, conditionally on $(\mathcal{T}_{\alpha}, \mu_{\alpha})$, let $(X_i,i\geq 1)$ be leaves sampled independently according to the measure $\mu_{\alpha}$.  Let 
\begin{equation}
\label{defmarginals}
\mathcal T_{\alpha,p}:=\cup_{0\leq i,j \leq p}[[X_i,X_j]]
\end{equation} 
be the subtree of $\mathcal{T}_{\alpha}$ spanned by the root $X_0$ and the $p$ leaves $X_1,\ldots,X_p$, for $p \geq 1$ (here, $[[X_i,X_j]]$ denotes the geodesic line-segment between $X_i$ and $X_j$). 
It turns out that the sample of leaves $(X_i,i\geq 1)$ is dense in $\mathcal T_{\alpha}$. Therefore, we can almost surely recover $\mathcal T_{\alpha}$ as the completion of the increasing union
$\cup_{p \geq 1} \mathcal T_{\alpha,p}$. In particular, the laws of the trees $(\mathcal{T}_{\alpha,p},p \ge 1)$ are sufficient to specify the law of $\mathcal{T}_{\alpha}$.  
It is this last perspective which we will pursue in this paper.

In the case of the Brownian CRT, the subtrees spanned by the root and $p$ leaves chosen independently at random have a particularly beautiful description, due to Aldous~\cite{AldousCRT1}, in terms of a Poisson line-breaking construction.  We give a somewhat informal presentation for ease of comprehension; this can be readily turned into a formal metric space construction.  Consider an inhomogeneous Poisson process on $\R_+$ of intensity measure $t \d t$.  Let $R_0 = 0$ and let $R_1, R_2, \ldots$ be the locations of the points of the process in increasing order.  For $k \ge 1$, let $S_k = R_k - R_{k-1}$.  Now proceed to build trees $(\mathcal{T}_p,p \ge 1)$ as follows.

\begin{center}
\fbox{
\mbox{\begin{minipage}[t]{0.85\textwidth}
\textsc{Aldous' line-breaking construction}
\begin{itemize}
\item $\mathcal{T}_1$ consists of a closed line-segment of length $S_1$ (rooted at one end).
\item For $p \ge 2$, take a closed line-segment of length $S_p$ and glue it onto $\mathcal{T}_{p-1}$ at a point chosen uniformly (i.e. according to the normalized Lebesgue measure on $\mathcal T_{p-1}$) at random.
\end{itemize}
\end{minipage}}
}
\end{center}

This yields a nested sequence of rooted $\R$-trees $(\mathcal{T}_p,p \ge 1)$ which have the same joint distribution as $\big(\tfrac{1}{\sqrt{2}}  \cdot \mathcal{T}_{2,p},p \ge 1\big)$ and which thus leads to a version of $\tfrac{1}{\sqrt{2}}\cdot \mathcal{T}_2$.  

The purpose of this paper is to show that there exists an analogous line-breaking construction of the $\alpha$-stable tree for a general $\alpha \in (1,2)$. This construction is necessarily a little more complex, since it also glues segments to branch-points.
We refer to Duquesne and Winkel \cite{DuW07} and Le Jan \cite{LJ91} for other tree-growth procedures yielding stable trees, via consistent Bernoulli percolation on Galton-Watson trees with edge-lengths in \cite{DuW07} and via erasing small parts of the tree near the leaves in \cite{LJ91}.  Marchal \cite{Marchal} gives an algorithm which generates a sequence of rooted discrete trees $(\tilde T_p)_{p \ge 1}$ which has the same distribution as the sequence of the combinatorial \emph{shapes} of the subtrees $\mathcal{T}_{\alpha,p}$ of $\mathcal{T}_{\alpha}$, $p\geq 1$. We will recover Marchal's algorithm as an aspect of ours.  The novelty of our approach is that we will construct an increasing sequence of trees $(\mathcal{T}_p,p \ge 1)$ having the same distribution as $(\mathcal{T}_{\alpha,p},p \ge 1)$ directly.

The rest of this paper is structured as follows.  In Section~\ref{sec:main}, we first introduce some of the main distributional ingredients of our construction, before defining a Markov chain on $\R_+$ which plays the analogous role to Aldous' inhomogeneous Poisson process.  We then give two versions of our line-breaking construction, and shed new light on several distributional properties of the tree $\mathcal T_{\alpha}$.  In Section~\ref{sec:DR}, we state and prove various results concerning generalized Mittag-Leffler and Dirichlet distributions, which play a key part in our proofs.  The proofs of the results in Section~\ref{sec:main} may then be found in Section~\ref{sec:proofs}.  We conclude the paper in Section~\ref{sec:complements} with complements and connections to various known results in the literature.

\section{Main result}
\label{sec:main}

In order to fix notation, we briefly recall the definitions of some basic distributions which we will use in the rest of the paper.  The Gamma distribution with parameters $\gamma > 0$ and $\lambda > 0$ has density
\[
\frac{\lambda^{\gamma}}{\Gamma(\gamma)} x^{\gamma-1} e^{-\lambda x}
\]
with respect to the Lebesgue measure on $(0,\infty)$.  Since we will always take $\lambda=1$ in the following, we will write $\mathrm{Gamma}(\gamma)$ for this distribution.  For $a,b >0$, the $\mathrm{Beta}(a,b)$ distribution has density
\[
\frac{\Gamma(a+b)}{\Gamma(a) \Gamma(b)} x^{a-1} (1-x)^{b-1}
\]
with respect to the Lebesgue measure on $(0,1)$. For parameters $a_1, a_2, \ldots, a_n > 0$, the Dirichlet distribution $\mathrm{Dir}(a_1, a_2, \ldots, a_n)$ has density
\[
\frac{\Gamma(\sum_{i=1}^n a_i)}{\prod_{i=1}^n \Gamma(a_i)} \prod_{j=1}^{n} x_i^{a_j-1}
\]
with respect to the Lebesgue measure on $\{(x_1, \ldots, x_n) \in [0,1]^n: \sum_{i=1}^n x_i = 1\}$.

\subsection{The generalized Mittag-Leffler distribution}

Let us now introduce the generalized Mittag-Leffler distribution, which will play an important part in the sequel.  For $\beta \in (0,1)$, let $\sigma_{\beta}$ be a stable random variable having Laplace transform 
\begin{equation}
\label{stablerv}
\mathbb E\big[e^{-\lambda \sigma_{\beta}}\big] =  \exp(-\lambda^{\beta}), \quad \lambda \geq 0.
\end{equation}
Following Pitman~\cite{PitmanStFl}, for $0<\beta<1$, $\theta>-\beta$, we say that a random variable $M$ has the \emph{generalized Mittag-Leffler distribution} $\mathrm{ML}(\beta,\theta)$ if, for all suitable test functions $f$,
\begin{equation}
\label{def:ML}
\mathbb E\left[ f(M)\right]= C_{\beta,\theta} \mathbb E\left[ \sigma_{\beta}^{-\theta}f\big(\sigma_{\beta}^{-\beta}\big)\right],
\end{equation}
where $C_{\beta,\theta}$ is the appropriate normalizing constant. (The distribution $\mathrm{ML}(\beta,0)$ is the Mittag-Leffler distribution with parameter $\beta$, where the name derives from the fact that the moment generating function of this distribution is the usual Mittag-Leffler function. Note that there are \emph{two} distributions referred to by this name: the other has the Mittag-Leffler function implicated in the definition of its cumulative distribution function rather than its moment generating function; see, for example, Pillai~\cite{Pillai}.)

Let $g_{\beta}$ be the density of $\sigma_{\beta}^{-\beta}$, so that $\mathrm{ML}(\beta,\theta)$ has density 
\[
g_{\beta,\theta}(t)=\frac{\Gamma(\theta+1)}{\Gamma(\theta/\beta+1)} t^{\theta/\beta} g_{\beta}(t), \quad t \ge 0.
\]
See Section 0.5 of Pitman~\cite{PitmanStFl} for an expression for $g_{\beta}$, but note that it has a simple form only when $\beta=1/2$, in which case, $$g_{1/2}(t) = \frac{1}{\sqrt{\pi}}\exp(-t^2/4), \quad t \geq 0.$$ This entails that $\mathrm{ML}(1/2,p-1/2)$ has density
\[
g_{1/2,p-1/2}(t) = \frac{1}{2^{2p-1} \Gamma(p)} t^{2p-1} \exp(-t^2/4), \quad t\ge 0,
\]
and shows, in particular, that $\mathrm{ML}(1/2,p-1/2) \overset{\mathrm{(d)}}= 2 \sqrt{\mathrm{Gamma}(p)}$ (with a slight abuse of notation).

\subsection{A Markov chain}

A key element in our line-breaking construction is an increasing $\R_+$-valued Markov chain $(M_p, p \ge 1)$, which will play a role similar to that of the inhomogeneous Poisson process in the Brownian case. The process $(M_p,p\geq 1)$ has two principal properties. Firstly, for each $p\geq 1$, we have
\begin{equation*}
\label{eq:loiMp}
M_p \sim \mathrm{ML}\left(1-1/\alpha,p-1/\alpha\right).
\end{equation*}
Secondly, for $p \ge 1$, the backward transition from $M_{p+1}$
to $M_{p}$ is given by
\begin{equation}
\label{eq:MLBeta}
M_{p}=M_{p+1} \cdot \beta_p,
\end{equation}
where $\beta_p$ is independent of $M_{p+1}$ and 
\begin{equation}
\label{betap}
\beta_p \sim \mathrm{Beta}\left( \frac{(p+1)\alpha - 2}{\alpha-1}, \frac{1}{\alpha-1} \right).
\end{equation}
Given that it is Markovian, these two properties completely characterize the distribution of the process $(M_p,p\geq 1)$.  It remains to check that such a Markov chain $(M_p, p\geq 1)$ indeed exists. We do so in the following lemma, which will be proved in Section~\ref{sec:MLD}. 

\begin{lem} 
\label{lem:alternativedefs}
$\mathrm{(i)}$ Let $\beta_1, \beta_2, \ldots$ be independent random variables such that $\beta_p$ has the Beta distribution (\ref{betap}).
Then, there exists a random variable $M_1  \sim \mathrm{ML}(1-1/\alpha, 1-1/\alpha)$ such that 
\[
\frac{\alpha}{\alpha-1} e^{-1/\alpha} n^{1/\alpha} \prod_{i=1}^{n-1} \beta_i \underset{n\rightarrow \infty}{\overset{\mathrm{a.s.}} \longrightarrow }M_1.
\]
Moreover, for $p \ge 2$, $$M_p: = \frac{M_1}{\beta_1 \ldots \beta_{p-1}}$$ has $\mathrm{ML}(1 - 1/\alpha, p - 1/\alpha)$ distribution, independently of $\beta_1, \ldots, \beta_{p-1}$.

$\mathrm{(ii)}$ The process $(M_p,p\geq 1)$ is a time-homogeneous Markov chain with transition density from state $m$ to state $m'$ given by
\begin{equation}
\label{eq:transition}
p(m,m') = \frac{(m'-m)^{\frac{2-\alpha}{\alpha-1}} m' g_{1-1/\alpha}(m')}{\alpha \Gamma(\frac{\alpha}{\alpha-1}) g_{1-1/\alpha}(m)}, \quad m' \ge m.
\end{equation}
\end{lem}

We observe that either (i), or (ii) taken together with $M_1 \sim \mathrm{ML}(1-1/\alpha,1-1/\alpha)$, could have served as our definition of the process $(M_p, p \ge 1)$. We now check that this process really is the analogue of (a constant times) the inhomogeneous Poisson process in Aldous' construction.

\begin{lem}
For $\alpha = 2$, $M_p, p \ge 1$ are the points, in increasing order, of an inhomogeneous Poisson process on $\R_+$ of intensity $t \d t/2$.
\end{lem}

\begin{proof}
It suffices to show that $M^2_p/4,p\geq 1$ are the points of a standard Poisson process on $\mathbb R_+$. From the expression (\ref{eq:transition}) for the transition probabilities and the fact that $g_{1/2}(t) = \pi^{-1/2} \exp(-t^2/4)$, we immediately see that $M^2_{p+1}-M^2_p$ is independent of $M_p^2$ and has density proportional to $\exp(-s/4)$. 
\end{proof}

\subsection{Line-breaking construction}

\textbf{Rooted trees with edge-lengths.} Henceforth, we will mainly work with rooted (finite) discrete trees with edge-lengths. We think of such an object as being defined by (1) its combinatorial structure, which is a finite connected acyclic (unlabelled) graph \textit{with no vertices of degree 2} and a distinguished vertex called the root, and (2) a sequence of lengths (positive real numbers) indexed by the set of edges of the combinatorial structure. We will call the combinatorial structure of a tree with edge-lengths its \emph{shape}.

Such a tree can naturally be turned into an  $\mathbb R$-tree by viewing its edges as line-segments. Reciprocally, it is obvious that any rooted $\mathbb R$-tree with a finite number of leaves and a root of degree 1 can be seen as a tree with edge-lengths as defined above, in a unique manner. In the following, we will use these two formalisms interchangeably.  

\textbf{Two (equivalent) line-breaking algorithms.} We now use the Markov chain $(M_p,p\geq 1)$ to construct an increasing sequence of rooted trees with edge-lengths $(\mathcal T_p, p \geq 1)$. For all $p \ge 1$, $\mathcal T_p$ will be a tree with $p$ leaves. We let $T_p$ denote its shape and $|T_p|$ its number of non-root vertices.  Finally, we will let $L_p$ denote the total length of $\mathcal T_p$, that is the sum of the lengths of its edges.  Observe that the distribution $\mathrm{Beta}(1,(2-\alpha)/(1-\alpha))$ converges weakly to a point mass at 1 as $\alpha \uparrow 2$.  We will, therefore, use the convention that $B \sim \mathrm{Beta}(1,0)$ means $B = 1$ almost surely.

\begin{center}
\fbox{
\mbox{\begin{minipage}[t]{0.85\textwidth}
\textsc{Line-breaking construction of a stable tree (I)}
\begin{itemize}
\item Start with $M_1$ and set $L_1=M_1$. Let $\mathcal T_1$ be the tree consisting of a closed line-segment of length $L_1$ (rooted at one end).
\item For $p\geq 1$, given $\mathcal T_p$ :
\begin{enumerate}
\item Let $B \sim \text{Beta}(1,\frac{2-\alpha}{\alpha-1})$ be independent of everything we have already constructed. We will glue a new branch (i.e. a closed line-segment) of length $(M_{p+1}-M_p)\cdot B$ onto $\mathcal T_p$, at a point to be specified.
\item In order to find where to glue the new branch, we first select either the set of edges of $T_p$, with probability $L_p/M_p$, or the set of vertices of $T_p$, with the complementary probability.
\item If we select the edges in 2., glue the new branch at a point chosen according to the normalized Lebesgue measure on $\mathcal T_p$.
\item If we select the vertices in 2., pick a vertex at random in such a way that a vertex of degree $d \geq 3$ is chosen with probability  $$\frac{d-1-\alpha}{p\alpha-1-|T_p|(\alpha-1)},$$ (Note that vertices of degree 1 or 2 cannot be selected.)  Then glue the new branch to the selected vertex. 
\end{enumerate}
\end{itemize} 
\end{minipage}
}
}
\end{center}

\begin{rem}
In the case $\alpha = 2$, we recover the line-breaking construction of Aldous (up to a constant scaling factor), since then $B\equiv1$, hence $L_p=M_p$, and the new branch is always glued uniformly at random to the pre-existing tree.  (In particular, no vertices of degree exceeding 3 are ever created.)
\end{rem}

\begin{rem}
\label{rem:somme}
When $1<\alpha<2$ and $p \geq 2$,
it is not hard to check that $|T_p|(\alpha-1)<p\alpha-1$, since $|T_p|\leq 3+2(p-2)$, and then inductively on $p \geq 2$ we have
$$\sum_{d \geq 3} (d-1-\alpha) \#\{\mathrm{vertices \ of \ degree \ } d \mathrm{\ in \ } T_p\}=p\alpha-1-|T_p|(\alpha-1).$$
Hence, the process whereby we select vertices is well-defined. 
\end{rem}

Note that in this algorithm the total length $L_{p+1}$ of $\mathcal T_{p+1}$ is implicitly given by
$$
L_{p+1} = L_p +  (M_{p+1} - M_p) \cdot B,
$$
where $B$ is the Beta random variable introduced in step 1. Also implicit in this construction is a random weight $M_p - L_p$ which, once renormalized by $M_p$, gives the probability of picking the set of vertices of $T_p$ rather than its set of edges.  It is also possible to think of assigning individual random weights to each of the vertices of $T_p$ in such a way that they sum to $M_p - L_p$ and the relative weight of a vertex gives the probability that it is picked. This is the second version of our algorithm. In contrast to the first version, the new branch is now attached to a vertex selected with a probability proportional to its weight rather than just depending on its degree.   In the following, we denote by $\{W^{(p)}_{v}, v \in \mathcal I_p\}$ the weights of the internal vertices of $T_p$, where $\mathcal I_p$ denotes this set of internal vertices.

\medskip
\begin{center}
\fbox{
\mbox{\begin{minipage}[t]{0.85\textwidth}
\textsc{Line-breaking construction of a stable tree (II)}
\begin{itemize}
\item Start with $M_1$ and set $L_1=M_1$. Let $\mathcal T_1$ be the tree consisting of  a closed line-segment of length $L_1$ (rooted at one end).
\item For $p\geq 1$, given $\mathcal T_p$ :
\begin{enumerate}
\item Let $B \sim \text{Beta}(1,\frac{2-\alpha}{\alpha-1})$ be independent of everything we have already constructed. We will glue a new branch of length $(M_{p+1}-M_p)\cdot B$  onto $\mathcal T_p$, at a position to be specified.
\item In order to find where to glue the new branch, select the set of edges of $T_p$ with probability $L_p/M_p$, or the internal vertex $v \in \mathcal I_p$ with probability $W^{(p)}_{v}/M_p$.
\item If we selected the edges in 2., then glue the new branch at a point chosen according to the normalized Lebesgue measure on $\mathcal T_p$ and assign the new internal vertex weight $(M_{p+1}-M_p)\cdot(1- B)$.
\item If we selected the internal vertex $v$ in 2., glue the new branch to it and add $(M_{p+1}-M_p)\cdot(1- B)$ to its weight, i.e. set $$W^{(p+1)}_{v}:=W^{(p)}_{v}+(M_{p+1}-M_p)\cdot(1- B).$$
\end{enumerate} 
\end{itemize} 
\end{minipage}
}
}
\end{center}

Note that the sum of the elements of $\{W_{v}^{(p)}, v \in \mathcal I_p\}$ is indeed $M_p-L_p$. In Section \ref{sec:proofs}, we will prove that the increasing sequences of trees built from these two algorithms have the same distributions. Now recall that $(\mathcal T_{\alpha,p},p\geq 1)$ denotes the increasing sequence of finite-dimensional marginals of $\mathcal T_{\alpha}$, as defined in (\ref{defmarginals}).  Our main result is the following.

\begin{thm} 
\label{thm:main}
The sequences of trees  $\left(\mathcal T_p,p\geq 1\right)$ and $\left(\mathcal T_{\alpha,p},p\geq 1\right)$ have the same distribution. As a consequence, a version of the stable tree $\mathcal T_{\alpha}$ is obtained as the completion of the union $\cup_{p\geq 1} \mathcal T_p$. \end{thm}

This theorem will be proved in Section \ref{sec:proofs}.

In the next proposition, we use intermediate results established in the proof of this theorem to elucidate the distributions of the lengths in $\mathcal T_{\alpha,p}$ (in both unconditional and conditional versions). Parts of these results can be either deduced from work of Duquesne and Le Gall~\cite{DuquesneLeGall} (see Theorem~\ref{thm:DuLG} and Proposition \ref{prop:lengths}  below) or other work on stable trees (see Section \ref{sec:complements}). In the following, we write $T_{\alpha,p}$ for the tree-shape of $\mathcal T_{\alpha,p}$.

\begin{prop}\label{cor:distns}
$\mathrm{(i)}$ Let $\mathrm t$ be a discrete rooted tree with $p \geq 2$ leaves. Then, conditionally on $T_{\alpha,p}= \mathrm t$, 
the sequence of edge-lengths of $\mathcal T_{\alpha,p}$ is distributed as 
$$M_p \cdot B_{|\mathrm t|} \cdot (D_1,\ldots ,D_{ |\mathrm t|}),$$  
these three random variables being independent, with $M_p \sim \mathrm{ML}(1 - 1/\alpha, p - 1/\alpha)$, \\ $B_{ |\mathrm t|} \sim \mathrm{Beta}\big(|\mathrm t|, \frac{p\alpha - 1}{\alpha - 1} - |\mathrm t| \big)$  and $(D_1,\ldots,D_{ |\mathrm t|}) \sim \mathrm{Dir}(1,\ldots,1)$.

$\mathrm{(ii)}$ The total length of the tree $\mathcal T_{\alpha,p}$,  conditionally on $|T_{\alpha,p}|=q$, has the same distribution as $$M_p \cdot B_q$$ where $M_p \sim \mathrm{ML}(1 - 1/\alpha, p - 1/\alpha)$ and $B_q \sim \mathrm{Beta}\big(q, \frac{p\alpha - 1}{\alpha - 1} - q \big)$ are independent.

$\mathrm{(iii)}$ The total length of the tree $\mathcal T_{\alpha,p}$ has the same distribution as
\begin{equation}
\label{totallength}
M_p \cdot \left(\prod_{j=1}^{p-1} \beta_j + \sum_{i=1}^{p-1} B_i(1-\beta_i)\prod_{j=i+1}^{p-1} \beta_j \right)
\end{equation}
where all of the random variables in this expression are independent, $B_1,\ldots,B_{p-1}$ are distributed as $\mathrm{Beta}(1,\frac{2-\alpha}{\alpha-1})$ and the $\beta_i, i \geq 1$ are those defined in (\ref{betap}).
\end{prop}
(Note that we use the convention that the sum over the empty set is 0, whereas the product over the empty set is 1).

We conclude this section with some observations about our two algorithms.

\begin{rem}
\label{rem:independence}
The edge-lengths and weights $W^{(p)}$ of $\mathcal T_p$, all divided by $M_p$, are independent of $(M_p, M_{p+1}, \ldots)$. Consequently, the chain $(M_p, M_{p+1}, \ldots)$ is independent of the shapes $T_1,\ldots,T_{p+1}$ (this follows from both Algorithms (I) and (II)). This observation will be useful later. Furthermore, using Lemma~\ref{lem:alternativedefs}, we note that the probability of selecting the set of edges has an almost sure limit:
$$\frac{L_p}{M_p} \underset{p \rightarrow \infty}{\overset{\mathrm{a.s.}}\longrightarrow} \mathbb E[B]=\alpha-1.$$
This can be seen by applying the strong law of large numbers for weighted averages (for example, Theorem 3 of \cite{JamisonOreyPruitt}).
\end{rem}

Finally, we note that there exists a very easy line-breaking construction of a \emph{randomly rescaled} version of the stable tree $\mathcal T_{\alpha}$, in the sense that the process used to break the half-line $\mathbb R_+$ is very simple. Indeed, let $X_1$ be a leaf of $\mathcal T_{\alpha}$ sampled according to $\mu_{\alpha}$ and let $\mathrm{ht}(X_1)$ denote its distance to the root. Then introduce
$$\mathcal T_{\alpha}^{\mathrm{norm}}:=\frac{1}{\mathrm{ht}(X_1)}\cdot  \mathcal T_{\alpha}$$
and consider the increasing process  $(\overline M_p,p\geq 1)$ defined by $\overline M_1=1$ and $\overline M_p=(\beta_1 \ldots \beta_{p-1})^{-1}$, for $p\geq 2$, where the sequence of independent Beta random variables $(\beta_p,p\geq 1)$ is that introduced in Lemma \ref{lem:alternativedefs}.

\begin{cor} If we run version \textsc{(I)} or \textsc{(II)} of the line-breaking algorithm with the sequence $(\overline M_p,p\geq 1)$ instead of $(M_p,p\geq 1)$, we obtain a sequence of trees $(\overline{\mathcal T}_p,p\geq 1)$ such that the completion of the increasing union $\cup_{p\geq 1} \overline{\mathcal T}_p$ is distributed as $\mathcal T_{\alpha}^{\mathrm{norm}}$.
\end{cor}

\section{Distributional relationships}
\label{sec:DR}

In this section we gather some elementary but useful results on generalized Mittag-Leffler and Dirichlet distributions. In particular, we prove Lemma \ref{lem:alternativedefs}.

\subsection{More on the generalized Mittag-Leffler distribution}
\label{sec:MLD}

Let $0<\beta<1$ and $\theta>-\beta$, recall the definition of the generalized Mittag-Leffler $\mathrm{ML}(\beta, \theta)$ distribution given in (\ref{def:ML}). From Pitman~\cite{PitmanStFl}, $\mathrm{ML}(\beta, \theta)$ has $k$th moment
\[
\frac{\Gamma(\theta) \Gamma(\theta/\beta + k)}{\Gamma(\theta/\beta) \Gamma(\theta + k \beta)}=\frac{\Gamma(\theta+1) \Gamma(\theta/\beta + k+1)}{\Gamma(\theta/\beta+1) \Gamma(\theta + k \beta+1)}
\]
and the collection of moments for $k \in \N$ uniquely characterizes this distribution. Using this fact, and observing that the $\mathrm{Beta}(a,b)$ distribution has $q$th moment
$$
\frac{\Gamma(a+b) \Gamma(a+q)}{\Gamma(a) \Gamma(a+b+q)}
$$
for $q \ge 0$, the distributional relationship implied by (\ref{eq:MLBeta}) is immediate.  The consideration of moments also gives a straightforward proof of the following useful characterization.

\begin{lem} \label{lem:GammaML}
Suppose that $G \sim \mathrm{Gamma}(\theta)$ and $M$ are independent.  Then $M \sim \mathrm{ML}(\beta, \theta)$ if and only if $G^{\beta} M \sim \mathrm{Gamma}(\beta/\theta)$.
\end{lem}

We will now prove Lemma~\ref{lem:alternativedefs}.

\begin{proof}[Proof of Lemma~\ref{lem:alternativedefs}]
(i) We first observe that if
\[
X_n := \frac{\Gamma(n+1-1/\alpha) \Gamma(2-2/\alpha)}{\Gamma(2-1/\alpha) \Gamma(n+1-2/\alpha)}\prod_{i=1}^{n-1} \beta_i 
\]
then $(X_n, n \ge 1)$ is a non-negative martingale of mean 1 in its natural filtration.  Indeed, 
%first note that for $k \ge 1$,
%\[
%\E{\beta_n^k} = \frac{\Gamma\left(\frac{(n+1)\alpha - 1}{\alpha-1} \right) \Gamma\left( \frac{(n+1)\alpha -2}{\alpha-1} + k \right)} {\Gamma \left(\frac{(n+1)\alpha - 2}{\alpha-1} \right) \Gamma \left( \frac{(n+1)\alpha -1}{\alpha-1} + k \right)}
%\]
%and so, in particular, 
%\[
%\E{\beta_n} = \frac{n+1 - 2/\alpha}{n+1 - 1/\alpha}.
%\]
the process $(X_n)_{n \ge 1}$ is clearly integrable,  we have $X_1 = 1$ and, by standard properties of the gamma function, for $n \ge 1$,
\[
X_{n+1} = \frac{n+1 -1/\alpha}{n+1 - 2/\alpha} \beta_n X_n.
\]
If $\mathcal{F}_n = \sigma(X_m: 1 \le m \le n)$ then
\[
\E{X_{n+1} | \mathcal{F}_n}  = \frac{n+1 -1/\alpha}{n+1 - 2/\alpha} \E{\beta_n} X_n = X_n.
\]
It follows from the martingale convergence theorem that $X_n \to X_{\infty}$ almost surely as $n \to \infty$, for some random variable $X_{\infty}$. By Stirling's approximation,
\[
\frac{\Gamma(n+1-1/\alpha)}{\Gamma(n+1-2/\alpha)} \sim e^{-1/\alpha} n^{1/\alpha}
\]
in the sense that the ratio of the left- and right-hand sides converges to 1 as $n \to \infty$.  Hence,
\[
\frac{\Gamma(2-2/\alpha)}{\Gamma(2-1/\alpha)} e^{-1/\alpha} n^{1/\alpha} \prod_{i=1}^{n-1} \beta_i \to X_{\infty}
\]
almost surely, as $n \to \infty$.  For $k \ge 2$, we have
\begin{align*}
\E{X_n^k}
 & = \left(\frac{\Gamma(n+1-1/\alpha) \Gamma(2-2/\alpha)}{\Gamma(2-1/\alpha) \Gamma(n+1-2/\alpha)}\right)^k \ \prod_{i=1}^{n-1} \frac{\Gamma\left(\frac{(i+1)\alpha - 1}{\alpha-1}\right) \Gamma\left(\frac{(i+1)\alpha-2}{\alpha-1} + k\right)} {\Gamma\left(\frac{(i+1)\alpha-2}{\alpha-1} \right) \Gamma\left(\frac{(i+1)\alpha-1}{\alpha-1} + k \right)} \\
%We observe that
%\[
%\frac{\Gamma\left(\frac{(i+1)\alpha - 1}{\alpha-1}\right) \Gamma\left(\frac{(i+1)\alpha-2}{\alpha-1} + k\right)} {\Gamma\left(\frac{(i+1)\alpha-2}{\alpha-1} \right) \Gamma\left(\frac{(i+1)\alpha-1}{\alpha-1} + k \right)} = \frac{\Gamma\left(\frac{(i+1)\alpha - 1}{\alpha-1} \right)}{\Gamma\left(\frac{i\alpha - 1}{\alpha-1} \right)} \frac{\Gamma\left( \frac{i \alpha - 1}{\alpha -1} + k \right)}{\Gamma\left(\frac{(i+1)\alpha-1}{\alpha-1} + k \right)} \frac{\frac{i \alpha - 1}{\alpha -1} + k}{\frac{i\alpha - 1}{\alpha-1}}
%\]
%and so much of the product telescopes to give
%\[
% \E{X_n^k} = \left(\frac{\Gamma(n+1-1/\alpha) \Gamma(2-2/\alpha)}{\Gamma(2-1/\alpha) \Gamma(n+1-2/\alpha)}\right)^k 
%\frac{\Gamma\left( \frac{n\alpha-1}{\alpha-1} \right) \Gamma(k+1)}{\Gamma\left( \frac{n\alpha -1}{\alpha-1} + k\right)} \prod_{i=1}^{n-1} \frac{i -1/\alpha + k(1 - 1/\alpha)}{i - 1/\alpha}.
%\]
%The last product term is equal to
%\[
%\frac{\Gamma(n-1/\alpha + k(1-1/\alpha)) \Gamma(1-1/\alpha)}{\Gamma((k+1)(1-1/\alpha)) \Gamma(n-1/\alpha)}
%\]
%and so
& = \frac{\Gamma(1-1/\alpha) k!}{\Gamma((k+1)(1 - 1/\alpha))}  \left( \frac{\Gamma(2 - 2/\alpha)}{\Gamma(1 - 1/\alpha)} \right)^k \\
& \qquad \times \left(\frac{\Gamma(n+1-1/\alpha)}{(1 - 1/\alpha) \Gamma(n+1-2/\alpha)}\right)^k \frac{\Gamma\left( \frac{n\alpha-1}{\alpha-1} \right) \Gamma(n - 1/\alpha + k(1-1/\alpha))}{\Gamma\left( \frac{n\alpha -1}{\alpha-1} + k\right)  \Gamma(n - 1/\alpha)},
\end{align*}
where, for the second equality, we have used the relation $\Gamma(a+1)=a\Gamma(a)$ and then that much of the product telescopes.
Thus, by another application of Stirling's approximation,
$$
 \E{X_n^k} \underset{n \rightarrow \infty} \longrightarrow \frac{\Gamma(1 - 1/\alpha) k!}{\Gamma((k+1)(1 - 1/\alpha))} \left( \frac{\Gamma(2 - 2/\alpha)}{\Gamma(1 - 1/\alpha)} \right)^k.
$$
Hence, $(X_n, n \ge 1)$ is bounded in $L^k$ for any $k \ge 1$ and so $X_n$ also converges in $L^k$.  Moreover,
\[
\E{ \left( \frac{\Gamma(1-1/\alpha)}{\Gamma(2 - 2/\alpha)} X_{\infty} \right)^k} =  \frac{\Gamma(1 - 1/\alpha) k!}{\Gamma((k+1)(1 - 1/\alpha)},
\]
and so $M_1 := \frac{\Gamma(1-1/\alpha)}{\Gamma(2 - 2/\alpha)} X_{\infty}  \sim \mathrm{ML}(1 - 1/\alpha,1-1/\alpha)$.  It follows that
\[
\frac{\alpha}{\alpha-1} e^{-1/\alpha} n^{1/\alpha} \prod_{i=1}^{n-1} \beta_i \to M_1
\]
almost surely, as required.

To conclude, we observe that $M_p:=M_1\prod_{i=1}^{p-1}\beta_i^{-1}$ is the almost sure limit of $\frac{\alpha}{\alpha-1} e^{-1/\alpha} n^{1/\alpha} \prod_{i=p}^{n-1} \beta_i$ and is, therefore, independent of $\prod_{i=1}^{p-1}\beta_i$. Its distribution follows straightforwardly by considering moments.

(ii) The Markov property follows from the fact that the distribution of $M_{p+1}=\beta_p^{-1}\cdot M_p$ conditional on $M_1,\ldots,M_p$ is the same as that of $\beta_p^{-1}\cdot M_p$ conditional on $M_p,\beta_1, \ldots,\beta_{p-1}$. This is, in turn, the same as the distribution of $\beta_p^{-1}\cdot M_p$ conditional on $M_p$, since $\beta_{p}$ is independent of $\beta_1, \ldots,\beta_{p-1}$.

For $p \ge 1$, $M_{p+1}$ has density 
\[
\frac{\Gamma(p+2 - 1/\alpha)}{\Gamma\left(\frac{\alpha(p+1) - 1}{\alpha-1} +1\right)} (m')^{\frac{\alpha(p+1) - 1}{\alpha-1}} g_{1-1/\alpha}(m')
\]
and $\beta_p$ has density 
\[
\frac{\Gamma\left(\frac{\alpha(p+1)-1}{\alpha-1}\right)}{\Gamma\left(\frac{\alpha (p+1) -2}{\alpha-1}\right) \Gamma\left(\frac{1}{\alpha-1}\right)} t^{\frac{\alpha p - 1}{\alpha-1}} (1-t)^{\frac{2- \alpha}{\alpha-1}}.
\]
We have $M_p = M_{p+1} \beta_p$, so let $m=m' t$ and change variables from $(t, m')$ to $(m, m')$; the Jacobian of this transformation is $1/m'$.  It follows that the joint distribution of $M_p$ and $M_{p+1}$ is 
\[
\frac{\Gamma\left(p+1-1/\alpha\right)}{\Gamma\left(\frac{\alpha p - 1}{\alpha-1} + 1 \right) \alpha \Gamma\left( \frac{\alpha}{\alpha-1}\right)}m^{\frac{\alpha p - 1}{\alpha-1}} (m'-m)^{\frac{2-\alpha}{\alpha-1}} m' g_{1 -1/\alpha}(m').
\]
Hence, the conditional distribution of $M_{p+1}$ given $M_p = m$ is equal to
\[
p(m,m') = \frac{(m'-m)^{\frac{2-\alpha}{\alpha-1}} m' g_{1 -1/\alpha}(m')}{\alpha \Gamma \left(\frac{\alpha}{\alpha-1}\right) g_{1 - 1/\alpha}(m)}. \qedhere
\]
\end{proof}

\subsection{Dirichlet distributions}
\label{sec:DirD}

We now recall a standard construction of the Dirichlet distribution, $\mathrm{Dir}(a_1,\ldots,a_n)$ for $a_1, a_2$, $\ldots, a_n > 0$.  Let $$\Gamma_{a_i} \sim \mathrm{Gamma}(a_i)$$ be independent for $1 \leq i \leq n$. Then, 
\begin{equation} \label{eqn:betagamma}
\left(\frac{\Gamma_{a_1}}{\sum_{i=1}^n \Gamma_{a_i}},\ldots, \frac{\Gamma_{a_n}}{\sum_{i=1}^n \Gamma_{a_i}} \right) \sim \mathrm{Dir}(a_1,\ldots,a_n) \text{ and } \sum_{i=1}^n \Gamma_{a_i} \sim \mathrm{Gamma}\left(\sum_{i=1}^n{a_i}\right),
\end{equation}
independently.  This is the main ingredient needed in the four following lemmas, which will be fundamental to the proof of Theorem \ref{thm:main}. Henceforth, $\Gamma_a$ will always denote a $\mathrm{Gamma}(a)$ random variable.

Note first that it is easy to see from (\ref{eqn:betagamma}) that if $(D_1,\ldots, D_n) \sim \mathrm{Dir}(a_1,\ldots,a_n)$, then $$(D_1+ D_2,D_3, \ldots,D_n) \sim \mathrm{Dir}(a_1+a_2,a_3,\ldots,a_n).$$  We state a standard result which appears as Lemma 17 of \cite{ABBrGo}.
\begin{lem}
\label{lem:biasDir}
Suppose that $(D_1, D_2, \ldots, D_n) \sim \mathrm{Dir}(a_1, a_2, \ldots, a_n)$.  Let $I$ be the index of a size-biased pick from amongst the co-ordinates: in other words,
\[
\Prob{I=i|D_1, D_2, \ldots, D_n} = D_i,
\]
for $1 \le i \le n$.  Then
\[
\Prob{I = i} = \frac{a_i}{a_1 + a_2 + \ldots + a_n}
\]
for $1 \le i \le n$ and, conditionally on $I=i$,
\[
(D_1, D_2, \ldots, D_n) \sim \mathrm{Dir}(a_1, \ldots, a_{i-1}, a_i + 1, a_{i+1}, \ldots, a_n).
\]
\end{lem}

The next lemma follows straightforwardly from (\ref{eqn:betagamma}). 

\begin{lem}
\label{lem:decompDir}
Suppose that $(D_1, D_2, \ldots, D_n) \sim \mathrm{Dir}(a_1, a_2, \ldots, a_n)$. Then, for $1 \leq p \leq n-1,$
$$
(D_1,\ldots,D_p)= B_p \cdot (\tilde D_1,\ldots, \tilde D_p),
$$
where $B_p \sim \mathrm{Beta}\big(\sum_{i=1}^p a_i,\sum_{i={p+1}}^n a_i\big)$ and $(\tilde D_1,\ldots, \tilde D_p) \sim \mathrm{Dir}(a_1,\ldots,a_p)$ are independent.
\end{lem}

The proofs of the two following lemmas are very similar. We only develop the first one.

\begin{lem}
\label{lem:recursion} Consider two integers $p,k \geq 1$ such that $p+1 \leq k \leq 2p-1$ and a strictly positive sequence $(a_1,\ldots, a_{k-p})$ such that
$$\sum_{i=1}^{k-p} a_i=\frac{(p+1)\alpha -2}{\alpha-1}-k.$$
Consider independent random variables $B \sim \mathrm{Beta}\big(1,\frac{2-\alpha}{\alpha-1}\big)$, $B_p \sim  \mathrm{Beta}\left(\frac{(p+1)\alpha-2}{\alpha-1},\frac{1}{\alpha-1}\right)$ and
$$(D_1,\ldots, D_k, D_{k+1}, \ldots, D_{2k-p}) \sim \mathrm{Dir}\big(\underbrace{1, \ldots, 1}_{k}, a_1,\ldots, a_{k-p}\big).$$
Then, for any $1 \leq i^{*} \leq k-p$, the random vector
\[
B_p \cdot \Bigg(D_1, \ldots,  D_k, \tfrac{\left(1- B_p\right) B}{B_p},   D_{k+1}, \ldots,   D_{k+i^{*}-1},   D_{k+i^{*}}+ \tfrac{\left(1-B_p\right)(1-B)}{B_p}, D_{k+i^{*}+1},\ldots,  D_{2k-p}\Bigg) 
\]
is distributed as
$$\mathrm{Dir}\bigg(\underbrace{1, \ldots, 1}_{k+1},a_1,\ldots, a_{i^{*}-1}, a_{i^{*}}+ \frac{2-\alpha}{\alpha-1}, a_{i^{*}+1}, \ldots, a_{k-p}\bigg).
$$
\end{lem}

\begin{proof} We use the construction of Dirichlet and Beta random variables via independent Gamma random variables. Let $\Gamma^{(i)}_1, 0 \leq i \leq k$, $\Gamma_{a_j}, 1 \leq j \leq k-p$ and $\Gamma_{\frac{2-\alpha}{\alpha-1}}$ be independent Gamma random variables and note that, without loss of generality, we may assume that
$$
B=\frac{\Gamma_1^{(0)}}{\Gamma_1^{(0)}+\Gamma_{\frac{2-\alpha}{\alpha-1}}}, \quad \quad  B_p=\frac{\sum_{i=1}^k \Gamma^{(i)}_1+\sum_{i=1}^{k-p} \Gamma_{a_i}}{\sum_{i=1}^k \Gamma^{(i)}_1+\sum_{i=1}^{k-p} \Gamma_{a_i}+\Gamma_1^{(0)}+\Gamma_{\frac{2-\alpha}{\alpha-1}}},$$ and $$D_j=\frac{\Gamma^{(j)}_1}{\sum_{i=1}^k \Gamma^{(i)}_1+\sum_{i=1}^{k-p} \Gamma_{a_i}}, \ 1\leq j \leq k, \quad \quad D_j=\frac{\Gamma_{a_{j-k}}}{\sum_{i=1}^k \Gamma^{(i)}_1+\sum_{i=1}^{k-p} \Gamma_{a_i}}, \ k+1\leq j \leq 2k-p.
$$
Let $\Gamma_{\mathrm{total}}=\sum_{i=1}^k \Gamma^{(i)}_1+\sum_{i=1}^{k-p} \Gamma_{a_i}+\Gamma_1^{(0)}+\Gamma_{\frac{2-\alpha}{\alpha-1}}$. This entails that
$$
B_p\cdot D_j = \frac{\Gamma_1^{(j)}}{\Gamma_{\mathrm{total}}}, \ 1 \leq j \leq k, \quad \quad B_p\cdot D_j =\frac{\Gamma_{a_{j-k}}}{\Gamma_{\mathrm{total}}}, \ k+1\leq j \leq 2k-p$$ 
and that
$$(1-B_p)\cdot B=\frac{\Gamma_1^{(0)}}{\Gamma_{\mathrm{total}}}, \quad \quad B_p\cdot D_{k+i^{*}}+(1-B_p)\cdot (1-B)=\frac{\Gamma_{a_{i^{*}}}+\Gamma_{\frac{2-\alpha}{\alpha-1}}}{\Gamma_{\mathrm{total}}}.$$
The result follows.
\end{proof}

The next result is proved similarly.

\begin{lem}
\label{lem:recursion2}
Consider two integers $p,k \geq 1$ such that $p \leq k \leq 2p-1$ and a strictly positive sequence $(a_1,\ldots, a_{k-p})$ such that
$$\sum_{i=1}^{k-p} a_i=\frac{p\alpha -1}{\alpha-1}-k,$$
with the convention that this sequence is empty if $k=p$.
Consider independent random variables $B \sim \mathrm{Beta}\left(1,\frac{2-\alpha}{\alpha-1}\right)$,  $B_p \sim \mathrm{Beta}\left(\frac{(p+1)\alpha-2}{\alpha-1},\frac{1}{\alpha-1}\right)$, $U\sim \mathrm U(0,1)$ and
$$(D_1,D_2,\ldots,D_k, D_{k+1}, \ldots, D_{2k-p}) \sim \mathrm{Dir}\big(2,1,\ldots,1, a_1,\ldots, a_{k-p}\big).$$
Then
$$
B_p \cdot \left(D_1  U, D_1  (1-U), D_2,  \ldots,   D_{k},  \frac{(1-B_p) B}{B_p}, D_{k+1}, \ldots, D_{2k-p}, \frac{(1-B_p)(1-B)}{B_p} \right)
$$ 
is distributed as
$$
\mathrm{Dir}\bigg(\underbrace{1, \ldots, 1}_{k+2}, a_1, \ldots, a_{k-p}, \frac{2-\alpha}{\alpha-1} \bigg).
$$
\end{lem}

\section{Edge-lengths, weights and shapes} \label{sec:proofs}

The role of this section is to prove Theorem \ref{thm:main} for $1<\alpha<2$. In order to do this, we will first compute the joint distribution of the sequences of lengths and weights appearing in version (\textsc{II}) of the algorithm. In particular, we will use this to check that the two versions of the algorithm are equivalent (Section \ref{sec:lw}). It will also entail that the distributions of the sequences of edge-lengths given the shape of the tree are the same for $\mathcal T_p$ and $\mathcal T_{\alpha,p}$ (Section \ref{sec:l}). Then we will check that the sequence of shapes of $\mathcal T_p,p\geq 1$ and of $\mathcal T_{\alpha,p},p \geq 1$ are also identically distributed (Section \ref{sec:s}), which will lead us to the identity in distribution of  $(\mathcal T_p,p\geq 1)$ and $(\mathcal T_{\alpha,p},p \geq 1)$ (Section \ref{sec:proofth}). 

We begin by recalling a result of Duquesne and Le Gall \cite{DuquesneLeGall} on the marginal distribution of the tree $\mathcal{T}_{\alpha,p}$.  To that end, write $L^{(\alpha,p)}_e$ for the length of the edge $e \in E(T_{\alpha,p})$, where $E(T_{\alpha,p})$ denotes the set of edges of $T_{\alpha,p}$. For every vertex $v$ of $T_{\alpha,p}$,  write $d_v$ for its degree.  We reformulate part of Theorem 3.3.3 of \cite{DuquesneLeGall} (adjusted to take into account the fact that our $\alpha$-stable tree is unordered and a factor $\alpha$ bigger than theirs). For convenience, we will sometimes use the notation $v \in \mathrm t$ to denote a vertex $v$ of a discrete tree $\mathrm t$, thereby identifying $\mathrm t$ and its vertex-set.

\begin{thm}[Duquesne and Le Gall] \label{thm:DuLG}
Let $\alpha \in (1,2)$ and $p \ge 1$.  Consider the law of $\mathcal{T}_{\alpha,p}$.  If $\mathrm t$ is a rooted discrete tree with $p$ labelled leaves then
\begin{equation} \label{eqn:comb}
\Prob{T_{\alpha,p} = \mathrm t} = \frac{ \prod_{v \in \mathrm t: d_v \ge 3} (\alpha-1)(2-\alpha) \ldots (d_v - 1 - \alpha)}{(\alpha-1) (2\alpha-1)\ldots ((p-1)\alpha-1)}.
\end{equation}
Now fix $\mathrm t$ and let $e_i(\mathrm t), 1 \leq i \leq |\mathrm t|$ denote the sequence of its edges labelled arbitrarily. Conditionally on $T_{\alpha,p} =\mathrm t$, the edge-lengths $\big(L^{(\alpha,p)}_{e_i(\mathrm t)}, 1 \leq i \leq |\mathrm t|\big)$ have joint density
\[
\frac{\Gamma(p - 1/\alpha)}{\Gamma(p - (1-1/\alpha)|\mathrm t| - 1/\alpha)} \int_0^1 u^{p - 1- (1-1/\alpha)|\mathrm t| - 1/\alpha} q_{1-1/\alpha}\left( \sum_{i=1}^{|\mathrm t|} \ell_{e_i(\mathrm t)}, 1 - u \right) \mathrm d u,
\]
where $q_{1-1/\alpha}(s,\cdot)$ is the density of $s^{\alpha/(\alpha-1)}\sigma_{1-1/\alpha}$, as defined in (\ref{stablerv}). 
\end{thm}

Note that the distribution of the lengths, conditional on $T_{\alpha,p} =\mathrm t$, does not depend on the labelling of the edges of $\mathrm t$ i.e.\ the lengths are exchangeable. For this reason, we will simply write
$$
\big(L^{(\alpha,p)}_e, e \in E(T_{\alpha,p}) \big)
$$
to denote this sequence.

\subsection{Edge-lengths and weights of $\mathcal T_p$}
\label{sec:lw}

For the moment, we focus our attention on the sequence of trees $(\mathcal T_p, p\geq 1)$ built according to version (\textsc{II}) of our algorithm. Our goal is to find the joint distribution of the edge-lengths in the tree $\mathcal T_p$ and the weights on its internal vertices $\{W_{v}^{(p)},v \in \mathcal I_p\}$, given its shape $T_p$.

To do this, we must choose canonical labellings. We start with the edge-lengths, and proceed recursively as follows.
Firstly, we set $L^{(1)}_1=M_1$.
Then, given the vector $$L^{(p)}=\left(L_1^{(p)},\ldots,L_{|T_p|}^{(p)}\right)$$ of edge-lengths of $\mathcal T_p$, denote the length of the new edge added at step $p+1$ by $L^{(p+1)}_{|T_{p+1}|}$.
\begin{itemize}
\item If the new edge is added at internal vertex, note that $|T_{p+1}|=|T_{p}|+1$.  Set $L^{(p+1)}_j=L^{(p)}_j$ for $1 \leq i \leq |T_p|$.  
\item If the new edge is added to a point along an edge, we get $|T_{p+1}|=|T_{p}|+2$. If the edge chosen has length $L^{(p)}_i$, then this edge is split into two whose lengths are labelled $L^{(p+1)}_i$ and $L^{(p+1)}_{|T_p|+1}$ (with the rule that $L^{(p+1)}_i$ denotes the length of the edge closest to the root). The lengths of the other edges are unchanged, in the sense that $L^{(p+1)}_j=L^{(p)}_j$ for $1 \leq j \leq |T_p|, j\neq i$.
\end{itemize} 
By iterating $p$, this defines the vector $L^{(p)}$ (of length $|T_p|$) of edge-lengths of $\mathcal T_p$, for all $p\geq 1$. Of course, we get $L_p=\sum_{i=1}^{|T_p|} L_i^{(p)}$. We choose to put the elements of $\{W_{v}^{(p)},v \in \mathcal I_p\}$ in order of appearance of the internal vertices in the construction of $\mathcal T_p$. Let $W^{(p)}$ denote this ordered sequence of length $|T_p|-p$.  We denote by $d_1, \ldots, d_{|T_p|-p}$ the degrees of these internal vertices with the same ordering. The following result is the key point in our construction.

\begin{prop} \label{prop:constru2}
For $p \ge 2$, conditionally on the shapes $T_1,\ldots,T_p$, we have 
$$\Big(L^{(p)},W^{(p)}\Big)=M_p \cdot \left( Z^{(p)}_1, Z^{(p)}_2,\ldots, Z^{(p)}_{2|T_p|-p} \right),$$
where $M_p \sim \mathrm{ML}(1-1/\alpha, p-1/\alpha)$,
\[
Z^{(p)}:=\left(Z^{(p)}_1, Z^{(p)}_2, \ldots, Z^{(p)}_{2|T_p|-p}\right)
\sim \mathrm{Dir}\bigg(\underbrace{1, \ldots, 1}_{|T_p|}, \frac{d_1 - 1 - \alpha}{\alpha-1}, \ldots, \frac{d_{|T_p|-p} - 1 - \alpha}{\alpha - 1}\bigg) 
\]
and 
$Z^{(p)}$ is independent of $(M_p,M_{p+1},\ldots)$.
\end{prop}

Note that the Dirichlet distribution depends on $(T_1,\ldots,T_p)$ only through $T_p$. In particular the distribution of the rescaled vector of weights $W^{(p)}/M_p$ only depends on the degrees of the vertices of $T_p$.

\begin{proof}
The fact that $M_p$ still has a $\mathrm{ML}(1-1/\alpha, p-1/\alpha)$ distribution after conditioning on $T_1,\ldots, T_p$ is an immediate consequence of Remark \ref{rem:independence}.

For the rest of the proof, we proceed by induction on $p \geq 2$.  For $p=2$, we have $|T_p|=3$ and, independently of the Markov chain $M$, independent random variables $U\sim \mathrm{U}(0,1)$ and $B \sim \mathrm{Beta}(1, (2-\alpha)/(1-\alpha))$ such that $L_1^{(2)}=M_1  U$, $L_2^{(2)}=M_1 (1-U)$, $L_3^{(2)}=(M_2-M_1) B$ and $W_1^{(2)}=(M_2-M_1)(1-B)$. Then, 
$$
\big(L_1^{(2)}, L_2^{(2)}, L_3^{(2)}, W_1^{(2)} \big)= M_2 \cdot \big(\beta_1  U, \ \beta_1  (1-U), \ (1-\beta_1)  B, \ (1-\beta_1)  (1-B)\big),
$$
where the vector in parentheses on the right-hand side is independent of $(M_2,M_3,\ldots)$.  By Lemma \ref{lem:recursion2} (for $k=p=1$), this vector is distributed as $\mathrm{Dir}(1,1,1, (2-\alpha)/(\alpha-1))$. Hence the distribution of $(L_1^{(2)}, L_3^{(2)}, L_3^{(2)}, W_1^{(2)})$ is as claimed (conditionally on the shapes of $T_1$, $T_2$, since these are deterministic). 

Now suppose that the statement of the proposition holds for some integer $p\geq 2$. 
Then, conditionally on $T_1,\ldots,T_{p+1}$, there are two cases, depending on whether we obtain $T_{p+1}$ from $T_p$ by adding the new edge at a vertex or an edge.  We note that additionally conditioning on $T_{p+1}$ will change the distribution of $Z^{(p)}$, but not the fact that it is independent of $(M_p,M_{p+1},\ldots)$, by Remark \ref{rem:independence}. 

Suppose first that the new edge has been added at an internal vertex (in which case, $|T_{p+1}|=|T_p|+1$), say with degree $d_{i^*}$. By Lemma~\ref{lem:biasDir}, this has the effect of increasing the parameter corresponding to the vertex in the Dirichlet by 1. So, conditionally on this additional event, the distribution of $Z^{(p)}$ is now $$\mathrm{Dir}\bigg(\underbrace{1, \ldots, 1}_{|T_p|},\tfrac{d_1 - 1 - \alpha}{\alpha-1}, \ldots,  \tfrac{d_{i^*-1} - 1 - \alpha}{\alpha-1}, \tfrac{d_{i^*} - 1 - \alpha}{\alpha-1} +1, \tfrac{d_{i^*+1} - 1 - \alpha}{\alpha-1}, \ldots, \tfrac{d_{|T_p|-p} - 1 - \alpha}{\alpha - 1}\bigg).$$
 
Moreover, by definition, 
\begin{align*}
&\Big(L_1^{(p+1)},\ldots,L^{(p+1)}_{|T_{p+1}|},W^{(p+1)}_1,\ldots, W^{(p+1)}_{|T_{p+1}|-{p+1}} \Big) \\
&\qquad \qquad =\Big(M_p  Z_1^{(p)},  \ldots,  M_p  Z_{|T_p|}^{(p)},  (M_{p+1}-M_p)  B, M_p  Z_{|T_p|+1}^{(p)}, \ldots,  M_p  Z_{|T_p|+i^*-1}^{(p)}, \\
& \ \qquad \qquad \qquad   M_p  Z_{|T_p|+i^*}^{(p)}+(M_{p+1}-M_p)(1-B), M_p  Z_{|T_p|+i^*+1}^{(p)},  \ldots,, M_p  Z_{2|T_p|-p}^{(p)}  \Big) \\
& \qquad \qquad = M_{p+1} \cdot \beta_p \cdot \Big( Z_1^{(p)},  \ldots ,  Z_{|T_p|}^{(p)}, \frac{(1-\beta_p)  B}{\beta_p},  Z_{|T_p|+1}^{(p)}, \ldots,   Z_{|T_p|+i^*-1}^{(p)},\\
& \qquad \qquad \qquad \qquad \qquad \ \quad Z_{|T_p|+i^*}^{(p)}+\frac{(1-\beta_p)(1-B)}{\beta_p},  Z_{|T_p|+i^*+1}^{(p)},  \ldots,  Z_{2|T_p|-p}^{(p)}\Big)
\end{align*}
for some random variable $B \sim \mathrm{Beta}(1,(2-\alpha)/(\alpha-1))$ which is independent of everything else. Recall that the vector $Z^{(p)}$ is independent of $\beta_p=M_{p}/M_{p+1}$.
Hence, by Lemma \ref{lem:recursion}, we have
\begin{align*}
&Z^{(p+1)}:=\beta_p \cdot \Big( Z_1^{(p)},  \ldots ,  Z_{|T_p|}^{(p)}, \frac{(1- \beta_p)  B}{\beta_p}, Z_{|T_p|+1}^{(p)}, \ldots,   Z_{|T_p|+i^*-1}^{(p)}, \\ 
&\qquad \qquad \qquad \quad \  Z_{|T_p|+i^*}^{(p)}+\frac{(1-\beta_p)(1-B)}{\beta_p},  Z_{|T_p|+i^*+1}^{(p)},  \ldots,  Z_{2|T_p|-p}^{(p)}\Big),
\end{align*}
which is distributed as
\[
\mathrm{Dir}\Big(\underbrace{1, \ldots, 1}_{|T_{p+1}|}, \frac{d_1 - 1 - \alpha}{\alpha-1}, \ldots,  \frac{d_{i^*-1} - 1 - \alpha}{\alpha-1}, \frac{d_{i^*} - 1 - \alpha}{\alpha-1} +1 + \frac{2-\alpha}{\alpha-1}, \ldots, \frac{d_{|T_p|-p} - 1 - \alpha}{\alpha - 1}\Big).
\]
This is the required distribution, since 
$
d_{i^*}+\alpha-1+2-\alpha=d_{i^*}+1,
$
which is indeed the degree of the selected vertex in $\mathcal T_{p+1}$.
Moreover, $Z^{(p+1)}$ is independent of $(M_{p+1},M_{p+2},\ldots)$ conditionally on $(T_1,\ldots,T_{p+1})$, since $Z^{(p)}, \beta_p$, $B$ and $(T_1,\ldots,T_{p+1})$ are independent of $(M_{p+1},M_{p+2},\ldots)$.

Suppose now that we pass from $T_p$ to $T_{p+1}$ by gluing the new edge to an existing one, say the edge with length $L^{(p)}_i$ (so that $|T_{p+1}|=|T_p|+2$). Then, conditionally on  this additional event, the distribution of $Z^{(p)}$ is now  $$\mathrm{Dir}\bigg(\underbrace{1, \ldots, 1}_{i-1},2,\underbrace{1, \ldots, 1}_{|T_p|-i},  \frac{d_1 - 1 - \alpha}{\alpha-1}, \ldots, \frac{d_{|T_p|-p} - 1 - \alpha}{\alpha - 1}\bigg),$$ by Lemma \ref{lem:biasDir}. 
Then, for some random variables $U \sim \mathrm U(0,1)$ and $B \sim \mathrm{Beta}(1,(2-\alpha)/(\alpha-1))$ independent of everything else,
\begin{align*}
&\Big(L_1^{(p+1)},\ldots ,L^{(p+1)}_{|T_{p+1}|},W_1^{(p+1)}, \ldots, W_{|T_{p+1}|-(p+1)}^{(p+1)}\Big) \\
&\qquad \qquad =\Big(M_p  Z_1^{(p)}, \ldots,  M_p Z_{i-1}^{(p)},  M_p  Z_i^{(p)}  U,  M_p  Z_{i+1}^{(p)},  \ldots,  M_p Z_{|T_p|}^{(p)},  M_p  Z_i^{(p)}  (1-U), \\ 
& \qquad \qquad \qquad  (M_{p+1}-M_p) B, M_p Z_{|T_p|+1}^{(p)}, \ldots M_p Z_{2|T_{p}|-p}^{(p)}, (M_{p+1}-M_p)(1-B) \Big) \\
& \qquad \qquad = M_{p+1} \cdot \beta_p \cdot  \Big(Z_1^{(p)}, \ldots,    Z_{i-1}^{(p)},  Z_i^{(p)}  U,   Z_{i+1}^{(p)},  \ldots,   Z_{|T_p|}^{(p)},  Z_i^{(p)}  (1-U), \\
& \qquad \qquad \qquad  \qquad  \qquad \quad \frac{(1-\beta_p)B}{\beta_p}, Z_{|T_p|+1}^{(p)}, \ldots M_p Z_{2|T_{p}|-p}^{(p)},   \frac{(1-\beta_p)(1-B)}{\beta_p}  \Big)\\
&\qquad \qquad =: M_{p+1} \cdot Z^{(p+1)}.
\end{align*}
The vector $ Z^{(p)}$ is independent of $\beta_p, U$ and $B$ and so, by Lemma \ref{lem:recursion2} and exchangeability, the distribution of $Z^{(p+1)}$ is  $$\mathrm{Dir}\bigg(\underbrace{1, \ldots, 1}_{|T_p|+2}, \frac{d_1 - 1 - \alpha}{\alpha-1}, \ldots, \frac{d_{|T_p|-p} - 1 - \alpha}{\alpha - 1}, \frac{2-\alpha}{\alpha-1}\bigg),$$
as required, since $|T_{p+1}|=|T_{p}|+2$ and the degree in $\mathcal T_{p+1}$ of the new vertex is 3.
Finally, $Z^{(p+1)}$ is independent of $(M_{p+1},M_{p+2},\ldots)$ conditionally on $(T_1,\ldots,T_{p+1})$, since $Z^{(p)}, \beta_p$, $B$, $U$ and $(T_1,\ldots,T_{p+1})$ are independent of $(M_{p+1},M_{p+2},\ldots)$.
\end{proof}

\begin{cor}
\label{cor:distribLW}
For $p\geq 2$, conditionally on the shapes $T_1,\ldots,T_p$, we have
$$
\Big(L_1^{(p)},\ldots,L^{(p)}_{|T_p|},M_p-L_p \Big)=M_p \cdot \Big(\tilde Z^{(p)}_1, \tilde Z^{(p)}_2,\ldots, \tilde Z^{(p)}_{|T_p|+1} \Big),
$$
where
$$\tilde Z^{(p)}:=\Big(\tilde Z^{(p)}_1, \tilde Z^{(p)}_2, \ldots, \tilde Z^{(p)}_{|T_p|+1}\Big) \sim \mathrm{Dir}\Big(\underbrace{1, \ldots, 1}_{|T_p|}, \frac{p\alpha- 1}{\alpha-1}-|T_p|\Big),$$ $Z^{(p)}$ is independent of $(M_p,M_{p+1},\ldots)$ and $M_p \sim \mathrm{ML}(1-1/\alpha, p-1/\alpha)$.
\end{cor}

This follows from Remark \ref{rem:somme}, the fact that $\sum_{v \in \mathcal I_p}W^{(p)}_v=M_p-L_p$ and the additive property of the Dirichlet distribution mentioned at the beginning of Section \ref{sec:DR}. Note that a consequence of this result is that the distribution of the sequence of lengths of $\mathcal T_p$, given its shape $T_p$, is exchangeable and depends only on $|T_p|$.  In particular it does not depend on how the lengths were labelled. 

In a similar manner, the following corollary is a consequence of Proposition~\ref{prop:constru2} and Lemma~\ref{lem:decompDir}.

\begin{cor} \label{cor:distribLW2}
For $p \ge 2$, conditionally on the shapes $T_1, \ldots, T_p$, we have
\[
\left(L_p, W^{(p)}\right) = M_p \cdot \left(P_p, \ (1-P_p) \cdot \left(\hat{Z}^{(p)}_1, \hat{Z}^{(p)}_2, \ldots, \hat{Z}^{(p)}_{|T_p| - p} \right) \right),
\]
where
\[
P_p  \sim \mathrm{Beta} \left(|T_p|, \frac{p \alpha - 1 - |T_p|(\alpha-1)}{\alpha-1} \right)
\]
and
\[
\left(\hat{Z}^{(p)}_1, \hat{Z}^{(p)}_2, \ldots, \hat{Z}^{(p)}_{|T_p| - p} \right)  \sim \mathrm{Dir} \left( \frac{d_1 - 1 - \alpha}{\alpha-1}, \ldots, \frac{d_{|T_p|-p} - 1 - \alpha}{\alpha - 1}\right)
\]
are mutually independent and independent of $(M_p, M_{p+1}, \ldots)$.
\end{cor}

Finally, we turn to the equivalence of the two versions of our construction.
\begin{prop}
The sequences of trees with edge-lengths generated by versions (I) and (II) of our construction have the same law.
\end{prop}
\begin{proof}
Consider the sequence $((T_p, (L_1^{(p)}, \ldots, L^{(p)}_{|T_p|}), \{W_{v}^{(p)}, v \in \mathcal I_p \}), p \ge 1)$ generated by version (II) of our construction.  This sequence evolves in a Markovian manner (as detailed above).  It suffices to show that the transition probabilities for the shape and lengths are as given in version (I) if we average over the vertex weights.  So fix $p \ge 1$ and a tree with edge-lengths $t_p$ (of shape $\mathrm t_p$), and consider the $(p+1)$th step of the construction, conditional on $\mathcal T_p = t_p$ with vertex degrees $d_v, v \in \mathrm{t}_p$, $(L^{(p)}_1, \ldots, L^{(p)}_{|\mathrm t_p|}) = (\ell_1, \ldots, \ell_{|\mathrm t_p|})$ and $M_p = m_p$.  The new length is generated in the same way in both algorithms, and it is clear that when we choose to glue the new branch to a pre-existing edge, we do the same in both versions of the construction. So suppose instead that we choose to glue the new branch to one of the vertices of $\mathrm t_p$, an event which occurs with probability $(m_p-\ell_p)/m_p$ in either construction. In version (II), conditionally on this event, we pick the vertex $v \in \mathrm{t}_p$ with probability $W_v^{(p)}/\sum_{w \in \mathrm t_p} W_w^{(p)}$.  But by Corollary~\ref{cor:distribLW2}, still conditionally on the event that the new branch will be attached to a vertex, the distribution of 
\[
\frac{1}{\sum_{v \in \mathrm t_p} W_v^{(p)}} (W_v^{(p)}, v \in \mathrm{t}_p)
\] 
is Dirichlet with parameters $\left(\frac{d_v - 1 - \alpha}{\alpha-1}, v \in \mathrm t_p \right)$.  But then by Lemma~\ref{lem:biasDir}, we see that when we average over the normalized weights, we pick vertex $v \in \mathrm{t}_p$ with probability $\frac{d_v - 1 - \alpha}{p\alpha - 1 - |\mathrm t_p|(\alpha-1)}$, as in version (I).  The result follows.
\end{proof}

\subsection{Identification of edge-lengths given the shape}
\label{sec:l}

As observed in the previous section, the sequences of trees obtained in versions (\textsc{I}) and (\textsc{II}) of the algorithm have the same distribution. From now on, we denote this sequence by $(\mathcal  T_p,p\geq 1)$, without specifying whether it is obtained using the first or second version of the algorithm, and similarly for the sequence of shapes $(T_p,p\geq 1)$. Also, as observed in Corollary \ref{cor:distribLW}, the sequence of lengths of $\mathcal T_p$ given $T_p$ is exchangeable. To be consistent with the notation for the lengths of $\mathcal T_{\alpha,p}$, we can (and will) therefore use the notation
$$
L^{(p)}=\big(L_e^{(p)}, e \in E(T_p) \big)
$$
to denote the sequence of edge-lengths of $\mathcal T_p$.

\begin{prop} 
\label{prop:lengths}
Let $\mathrm t$ be a discrete rooted tree with $p \geq 2$ leaves. Then, 
the sequence of lengths $\big(L_e^{(p)}, e \in E(T_p) \big)$ of $\mathcal T_p$ conditional on $T_p= \mathrm t$ and the sequence of lengths $\big(L^{(\alpha,p)}_e, e \in E(T_{\alpha,p}) \big)$ of $\mathcal T_{\alpha,p}$ conditional on $T_{\alpha,p}= \mathrm t$ have the same distribution, which is that of
$$M_p \cdot B_{|\mathrm t|} \cdot (D_1,\ldots ,D_{ |\mathrm t|}),$$ where the three random variables are independent, and $M_p \sim  \mathrm{ML}(1 - 1/\alpha, p - 1/\alpha)$, $B_{ |\mathrm t|} \sim \mathrm{Beta}\big(|\mathrm t|, \frac{p\alpha - 1}{\alpha - 1} - |\mathrm t| \big)$  and $(D_1,\ldots,D_{ |\mathrm t|}) \sim \mathrm{Dir}(1,\ldots,1)$.
\end{prop}

\begin{proof}
For $\mathcal T_p$, this is an immediate consequence of Corollary \ref{cor:distribLW} and Lemma \ref{lem:decompDir}. For $\mathcal T_{\alpha,p}$, we use Theorem~\ref{thm:DuLG}.  Conditionally on $T_{\alpha,p} = \mathrm t$, the joint density of the edge-lengths depends on $(L^{(\alpha,p)}_e, e \in E(\mathrm t))$ only through $\sum_{e \in E(\mathrm t)} L^{(\alpha,p)}_e$, where $E(\mathrm t)$ designs the set of edges of $\mathrm t$. This entails that the random vector 
\[
\frac{1}{\sum_{e \in E(\mathrm t)} L^{(\alpha,p)}_e} (L^{(\alpha,p)}_e, e \in E(\mathrm t))
\]
has $\mathrm{Dir}(1,1,\ldots,1)$ distribution and is independent of $\sum_{e \in E(\mathrm t)} L^{(\alpha,p)}_e$.  Moreover, $\sum_{e \in E(\mathrm t)} L^{(\alpha,p)}_e$ has density
\[
\frac{1}{\Gamma(|\mathrm t|)} \frac{\Gamma(p-1/\alpha)}{\Gamma(p - (1 - 1/\alpha)|\mathrm t| - 1/\alpha)}   x^{|\mathrm t|-1}  \int_0^1 u^{p - (1 - 1/\alpha)|\mathrm t| - 1/\alpha - 1} q_{1-1/\alpha}( x, 1-u) \d u.
\]
Mimicking the proof of Lemma 4 in \cite{Mier03}, we have that, for $k \ge 1$,
\[
\int_0^{\infty} x^{|\mathrm t|+k-1} q_{1-1/\alpha}(x, 1-u)\d x = \frac{\Gamma(|\mathrm t|+k)}{\Gamma((|\mathrm t|+k)(1 - 1/\alpha))} (1 - u)^{(|\mathrm t|+k)(1 - 1/\alpha) -1}
\]
and so
\begin{align*}
& \E{\Bigg(\sum_{v \in  E(T_{\alpha,p})} L^{(\alpha,p)}_e \Bigg)^k \Bigg| T_{\alpha,p} = \mathrm t} \\
&= \frac{\Gamma(p-1/\alpha)\Gamma(|\mathrm t|+k) \alpha^{-k} }{\Gamma(|\mathrm t|) } \int_0^1 \frac{u^{p - (1 - 1/\alpha)|\mathrm t| - 1/\alpha - 1} (1 - u)^{(|\mathrm t|+k)(1 - 1/\alpha) -1}}{\Gamma(p - (1 - 1/\alpha)|\mathrm t| - 1/\alpha)\Gamma((|\mathrm t|+k)(1 - 1/\alpha))} \d u \\
& = \frac{\Gamma(p-1/\alpha)\Gamma(|\mathrm t|+k) \alpha^{-k} }{\Gamma(|\mathrm t|) \Gamma(k(1 - 1/\alpha) + p - 1/\alpha)} \\
& = \E{B_{|\mathrm t|}^k}\E{M_p^k}.
\end{align*}
The result follows.
\end{proof}

\subsection{Identification of shapes}
\label{sec:s}

We will now observe that our construction is intimately related to earlier work of Marchal.  Using (\ref{eqn:comb}) in Theorem \ref{thm:DuLG}, Marchal~\cite{Marchal} gives an algorithm which generates a sequence of rooted discrete trees $(\tilde T_p)_{p \ge 1}$ having the same distribution as the sequence of shapes $(T_{\alpha,p},p \geq 1)$. 

\begin{center}
\fbox{
\mbox{\begin{minipage}[t]{0.85\textwidth}
\textsc{Marchal's algorithm}
\begin{itemize}
\item Start from a tree $\tilde T_1$ consisting of the root joined by a single edge to another vertex labelled $1$.
\item For $p \ge 1$, assign the edges of $\tilde T_p$ weight $\alpha - 1$ and any vertex of degree $d \ge 3$ weight $d-1-\alpha$. Assign weight 0 to vertices of degree 1 or 2.
\begin{enumerate}
\item Pick an edge or a vertex with probability proportional to its weight.
\item If an edge was chosen in 1., split the edge into two with a new vertex in the middle and add an edge from that vertex to a new leaf, labelled $p+1$.
\item If a vertex was chosen in 1., create a new edge from the vertex to a leaf labelled $p+1$.
\end{enumerate}
\end{itemize}
\end{minipage}}
}
\end{center}

The case $\alpha = 2$ of this algorithm generates a uniform random binary rooted tree with $p$ labelled leaves, and is due to R\'emy~\cite{Remy}.  In \cite{Marchal}, Marchal proved that $p^{1/\alpha-1} \tilde T_p \to \mathcal{T}_{\alpha}$  in the sense of convergence of random finite dimensional distributions; this was improved to convergence in probability for the Gromov-Hausdorff distance in \cite{HMPW} and then to almost sure convergence in \cite{CurienHaas}.  
It is not hard to see (by induction on $p$) that the total weight on the edges and vertices of $\tilde T_p$ is $p\alpha-1$.

\begin{prop} \label{prop:Marchal}
The sequence of shapes of trees $(T_p,p\geq 1)$ follows Marchal's algorithm.  As a consequence, $(T_p,p\geq 1)$ and $(T_{\alpha,p},p\geq 1)$ have the same distribution.
\end{prop}

\begin{proof}
We use the first version of the line-breaking construction.
Given $\mathcal T_1,\ldots, \mathcal T_p$,  the new edge will be attached to the edge of length $L_e^{(p)}$ with probability
$$
\frac{L^{(p)}_e}{M_p}
$$ 
and to a given internal vertex of degree $d \geq 3$ with probability 
$$
\frac{M_p-L_p}{M_p} \times \frac{d-1-\alpha}{p\alpha-1-|T_p|(\alpha-1)}.
$$
Applying Corollary \ref{cor:distribLW} and Lemma \ref{lem:biasDir}, we get that, given the shapes of trees $T_1,\ldots,T_p$, the new edge is attached to a given edge with probability 
$$
\frac{\alpha-1}{p\alpha-1}
$$
and is attached to a given vertex of degree $d \geq 3$  with probability
\[
\frac{p\alpha-1-|T_p|(\alpha-1)}{p\alpha-1}\times \frac{d-1-\alpha}{p\alpha-1-|T_p|(\alpha-1)}=\frac{d-1-\alpha}{p\alpha-1}. \qedhere
\]
\end{proof}

\subsection{Proofs of the main results}
\label{sec:proofth}

\begin{proof}[Proof of Theorem \ref{thm:main}]
As an immediate consequence of Propositions \ref{prop:lengths} and  \ref{prop:Marchal}, we see that the trees $\mathcal T_p$ and $\mathcal T_{\alpha,p}$ have the same distribution, for all $p\geq 1$. It remains to check that the joint distributions are the same. 
In order to do this, we introduce an auxiliary labelling: label the leaves of each of these trees by order of appearance (i.e.\ the leaf we remove from $\mathcal T_{p+1}$ to obtain $\mathcal T_p$ is labelled $p+1$; similarly for $\mathcal T_{\alpha,p+1}$ and $\mathcal T_{\alpha,p}$). We denote by $\mathcal T^{\mathrm{lab}}_p$, $\mathcal T^{\mathrm{lab}}_{\alpha,p}$, $T^{\mathrm{lab}}_p$ and $T^{\mathrm{lab}}_{\alpha,p}$ the resulting labelled trees, with and without edge-lengths respectively, for $p\geq 1$.  By Proposition~\ref{prop:Marchal}, we see that $T^{\mathrm{lab}}_p$ and $T^{\mathrm{lab}}_{\alpha,p}$ have the same distribution. Since, moreover, the distribution of the lengths does not depend on the labelling, by Proposition \ref{prop:lengths}, we see that $\mathcal T^{\mathrm{lab}}_p$ and $\mathcal T^{\mathrm{lab}}_{\alpha,p}$ have the same distribution. But, removing the leaf $p$ and the adjacent edge in $\mathcal T^{\mathrm{lab}}_p$, we obtain $\mathcal T^{\mathrm{lab}}_{p-1}$ and similarly for the stable marginals. Iterating this leaf-deletion process, we finally obtain
$$
\big(\mathcal T_{1},\ldots,\mathcal T_{p}\big) \overset{\mathrm{(d)}}= \big(\mathcal T_{\alpha,1},\ldots,\mathcal T_{\alpha,p}\big),
$$
for all $p\geq 1$.
\end{proof}

\begin{proof}[Proof of Proposition~\ref{cor:distns}] (i) is part of Proposition \ref{prop:lengths}. Summing the lengths, we immediately get (ii). Finally, we get from the algorithmic  construction of $\mathcal T_p$ that is total length is distributed as (\ref{totallength}), from which we deduce (iii), since $\mathcal T_{\alpha,p}$ and $\mathcal T_p$ have the same distribution. 
\end{proof}

\section{Complements and connections} \label{sec:complements}

The main goal of this section is to explain the presence of generalized Mittag-Leffler distributions in our line-breaking construction.

\subsection{The Mittag-Leffler and Poisson-Dirichlet distributions}

The generalized Mittag-Leffler distributions arise naturally in the context of urn models.  Let $\beta \in (0,1)$ and $\theta > 0$.  Consider a generalized P\'olya urn scheme, where we have two colours, black and white.  Suppose that we pick each colour with probability proportional to the total weight of that colour in the urn.  If we pick black, add $1/\beta$ to the black weight.  If we pick white, add $1/\beta-1$ to the black weight and 1 to the white weight.  Start from 0 weight on black and $\theta/\beta$ weight on white and let $Y_n$ be the weight of white at step $n$.  Then by Theorem 1.3(v) of Janson~\cite{Janson}, 
\[
n^{-\beta} Y_n \to W \quad \text{ a.s. as $n \to \infty$,}
\]
where $W \sim \mathrm{ML}(\beta,\theta)$. (Janson states this as a convergence in distribution but it is straightforward to see by an argument using the martingale convergence theorem that the convergence must be almost sure; see Theorem 1.7 of \cite{Janson} for the characterization of the limit $W$ via its moments.)

Consider now the Poisson-Dirichlet distribution $\mathrm{PD}(\beta,\theta)$ defined as follows.  For $i \ge 1$, let $B_i \sim \mathrm{Beta}(1-\beta, \theta + i \beta)$ independently.  Now let
\[
P_j = B_j \prod_{i=1}^{j-1} (1 - B_i)
\]
and let $(P_i^{\downarrow})_{i \ge 1}$ be the sequence $(P_i)_{i \ge 1}$ ranked in decreasing order.  Then $(P_i^{\downarrow})_{i \ge 1}$ has the $\mathrm{PD}(\beta,\theta)$ distribution.  It then turns out that the so-called \emph{$\beta$-diversity}, $W$, defined to be the following almost sure limit,
\[
W := \Gamma(1 - \beta) \lim_{i \to \infty} i (P_i^{\downarrow})^{\beta},
\]
has the $\mathrm{ML}(\beta,\theta)$ distribution.  A connection to the generalized P\'olya urn can be made via the Chinese restaurant process construction (see \cite{PitmanStFl}) of an exchangeable random partition of $\N$ having $\mathrm{PD}(\beta,\theta)$ asymptotic frequencies.  There, if the number of tables occupied by the first $n$ customers is $K_n$, the $\beta$-diversity also arises as the almost sure limit
\[
W = \lim_{n \to \infty} n^{-\beta} K_n.
\]
Indeed, the number of tables (plus $1-\theta/\beta$, but this difference vanishes in the limit) evolves precisely according to the generalized P\'olya urn discussed above.

This means that we can think of results about $\mathrm{ML}(\beta,\theta)$ as results about $\mathrm{PD}(\beta,\theta)$ and vice versa.

\subsection{Masses and lengths in the stable trees}

The natural scaling relation for the $\alpha$-stable tree states that if we rescale the total mass by $x \in \R_+$, we must rescale distances by $x^{1-1/\alpha}$.  The self-similarity of the tree means that there are various natural ways to split it into subtrees such that the subtrees are randomly rescaled independent stable trees (see Miermont~\cite{Mier03,Mier05} for the first work on this topic).  One way to do this is using the \emph{spinal partitions} studied by
Haas, Pitman and Winkel~\cite{HaasPitmanWinkel}.  Consider the path between the root and a uniformly-chosen point (the \emph{spine}).  If we remove all of the points of degree 1 or 2 from this path, we split the tree into a forest where each tree has infinite degree at its root.  The masses of these trees form a random partition.  If we keep them in the order they arise along the branch, we obtain the \emph{coarse spinal interval partition}; if we rank them instead in decreasing order, we obtain the \emph{coarse spinal mass partition}.  If we remove the \emph{whole} path between the root and our uniformly-chosen point, we obtain a finer partition of mass which, when put in decreasing order, is called the \emph{fine spinal mass partition}.  We reproduce parts of Corollary 10 of \cite{HaasPitmanWinkel} which describe the distributions of these partitions.

\begin{thm}[Haas, Pitman \& Winkel] \label{thm:HPW}
Let $\alpha \in (1,2)$.  The following statements hold for the $\alpha$-stable tree $\mathcal{T}_{\alpha}$.
\begin{enumerate}
\item The distribution of the coarse spinal mass partition of $\mathcal{T}_{\alpha}$ is $\mathrm{PD}(1-1/\alpha,1-1/\alpha)$.
\item The coarse spinal interval partition is exchangeable.  Its $(1 -1/\alpha)$-diversity has the $\mathrm{ML}(1-1/\alpha, 1-1/\alpha)$ distribution and has the same distribution as the length of the spine.
\item The fine spinal mass partition is obtained from the coarse spinal mass partition by fragmenting every block with an independent $\mathrm{PD}(1/\alpha, 1/\alpha -1)$ random partition and then putting the blocks in decreasing order.
\item The (unconditional) distribution of the fine spinal mass partition is $\mathrm{PD}(1/\alpha, 1 - 1/\alpha)$.
\item Conditionally given that the fine spinal mass partition of $\mathcal{T}_{\alpha}$ is $(m_1, m_2, \ldots)$, the corresponding collection of subtrees obtained by removing the spine has the same distribution as $(m_1^{1-1/\alpha} \mathcal T^{(1)}, m_2^{1-1/\alpha} \mathcal T^{(2)}, \ldots)$, where $\mathcal T^{(1)}, \mathcal T^{(2)}, \ldots$ are independent copies of the $\alpha$-stable tree.
\end{enumerate}
\end{thm}

Notice that we should interpret the $(1-1/\alpha)$-diversity of the Poisson-Dirichlet partitions involved here as \emph{lengths} in the tree.  In a moment, we will use Theorem~\ref{thm:HPW} to go into more detail about the relationships between lengths and masses in $\mathcal{T}_{\alpha}$.  Before we do so, it will be useful to prove another distributional relationship for Dirichlet and Mittag-Leffler random variables.

\begin{prop} \label{prop:lengthsmasses}
Let $\beta \in (0,1)$ and $\theta > 0$.  For $n \ge 2$, let $\theta_1, \theta_2, \ldots, \theta_n > 0$ and such that $\sum_{i=1}^n \theta_i = \theta$.  Let $M \sim \mathrm{ML}(\beta,\theta)$ and $(Z_1, Z_2, \ldots, Z_n)  \sim \mathrm{Dir}\left(\frac{\theta_1}{\beta}, \ldots, \frac{\theta_n}{\beta}\right)$ independently. Let $(X_1, X_2, \ldots, X_n) \sim \mathrm{Dir}(\theta_1, \theta_2, \ldots, \theta_n)$ and $M^{(1)}, M^{(2)}, \ldots, M^{(n)}$ be independent random variables, where $M^{(i)} \sim \mathrm{ML}(\beta, \theta_i)$ for $1 \le i \le n$.
Then
\[
M \cdot \left( Z_1, Z_2, \ldots, Z_n \right) \equidist \left(X_1^{\beta} M^{(1)}, X_2^{\beta} M^{(2)}, \ldots, X_{n}^{\beta} M^{(n)} \right).
\]
\end{prop}
\begin{proof}
By (\ref{eqn:betagamma}), we may take
\[
(X_1, X_2, \ldots, X_{n}) = \frac{1}{\sum_{i=1}^{n} \Gamma_{\theta_i}} \left(\Gamma_{\theta_1}, \ldots, \Gamma_{\theta_{n}} \right),
\]
and the normalized vector is independent of 
\[
\Gamma_{\theta} := \sum_{i=1}^{n} \Gamma_{\theta_i} \sim \mathrm{Gamma}(\theta).
\]
By Lemma~\ref{lem:GammaML},
\[
\Gamma_{\theta}^{\beta} \cdot \left(X_1^{\beta} M^{(1)}, X_2^{\beta} M^{(2)}, \ldots, X_{n}^{\beta} M^{(n)} \right)
\equidist \left( \tilde{\Gamma}_{\theta_1/\beta}, \ldots, \tilde{\Gamma}_{\theta_{n}/\beta} \right),
\]
where the right-hand side is composed of independent Gamma random variables with the given parameters. But then
\begin{align*}
 \left( \tilde{\Gamma}_{\theta_1/\beta}, \ldots, \tilde{\Gamma}_{\theta_{n}/\beta} \right) & = \left( \sum_{i=1}^{n} \tilde{\Gamma}_{\theta_i/\beta} \right) \cdot \frac{1}{\sum_{i=1}^{n} \tilde{\Gamma}_{\theta_i/\beta}} \left( \tilde{\Gamma}_{\theta_1/\beta}, \ldots, \tilde{\Gamma}_{\theta_{n}/\beta} \right) \\
& \equidist \Gamma_{\theta}^{\beta} \cdot M \cdot \frac{1}{\sum_{i=1}^{n} \tilde{\Gamma}_{\theta_i/\beta}}  \left( \tilde{\Gamma}_{\theta_1/\beta}, \ldots, \tilde{\Gamma}_{\theta_{n}/\beta} \right) \equidist \Gamma_{\theta}^{\beta} \cdot M \cdot Z,
\end{align*}
where the three factors on the right-hand side are independent.  But then considering moments and applying the Cram\'er-Wold theorem, we can remove the factor $\Gamma_{\theta}^{\beta}$ on both sides of this identity in law to get the desired result.
\end{proof}

Note that this immediately yields a different representation for the joint distribution of the lengths and weights of the tree $\mathcal{T}_p$ given in Proposition~\ref{prop:constru2}.

We also observe the standard fact about Poisson-Dirichlet distributions that a size-biased pick from amongst the blocks of a $\mathrm{PD}(\beta,\theta)$ partition has $\mathrm{Beta}(1-\beta,\beta + \theta)$ distribution. We will now use these ingredients to give a sketch of a quite different proof that the edge-lengths in the stable tree have the same distribution as those generated by our construction.  (This is part of Proposition~\ref{prop:lengths}.)

\begin{prop}
\label{prop:otherproof}
Conditionally on the shapes $T_{\alpha,1}, \ldots, T_{\alpha,p}$, we have that the lengths $L^{(\alpha,p)} = (L^{(\alpha,p)}_e, e \in E(T_{\alpha,p}))$ (with an arbitrary labelling) have joint distribution
\[
L^{(\alpha,p)} \equidist M_p \cdot \left(Z_1^{(p)}, Z_2^{(p)}, \ldots, Z_{|T_{\alpha,p}|}^{(p)} \right),
\]
where $M_p \sim \mathrm{ML}(1 -1/\alpha, p-1/\alpha)$ and
\[
\left(Z_1^{(p)}, \ldots, Z_{|T_{\alpha,p}|}^{(p)}, 1 - \sum_{i=1}^{|T_{\alpha,p}|} Z_i^{(p)}\right) \sim \mathrm{Dir}\left(1, \ldots, 1, \frac{p\alpha-1}{\alpha-1} - |T_{\alpha,p}| \right)
\]
are independent.
\end{prop}

\begin{proof}[Sketch proof]
When we pick a first uniform leaf, clearly we simply pick out the spine, which has length distributed as $\mathrm{ML}(1-1/\alpha,1-1/\alpha)$.  Consider what happens at the second step of the line-breaking construction: we need to add the part of the tree which links a second uniform leaf to the spine.  When we pick this leaf, we select one of the subtrees which branch off the spine (in the coarse sense) in a size-biased manner (i.e.\ we pick a subtree of mass $m$ with probability $m$).  So we pick a size-biased block from the coarse spinal partition.  Since by Theorem~\ref{thm:HPW} the coarse spinal mass partition is exchangeable, this subtree sits at a uniform position amongst the subtrees branching off the spine.  Moreover, the coarse spinal mass partition is distributed as $\mathrm{PD}(1-1/\alpha,1-1/\alpha)$, and so the masses $X_1$ and $X_2$ of the collections of subtrees above and below the picked subtree along the spine and the mass at the vertex are such that $(X_1, X_2,1-X_1-X_2) \sim \mathrm{Dir}(1-1/\alpha,1-1/\alpha, 1/\alpha)$.  Moreover, if we split the tree by removing the picked subtree entirely, we get two $\alpha$-independent stable trees, rescaled to have total masses $X_1$ and $X_2$ respectively.  It follows that
\begin{equation} \label{eqn:MLRDE}
M \equidist X_1^{1-1/\alpha} M^{(1)} + X_2^{1-1/\alpha} M^{(2)},
\end{equation}
where $M^{(1)}, M^{(2)}$ are independent $\mathrm{ML}(1 - 1/\alpha, 1- 1/\alpha)$ random variables, independent of $(X_1, X_2)$.

The subtree branching off the spine at our chosen vertex is, of course, composed of infinitely many subtrees (in the sense of the fine spinal partition).  Since the one we are interested in contains a vertex picked according to the mass measure, it sits inside a size-biased one of these trees.  To get the masses of these subtrees, Theorem~\ref{thm:HPW} says that we need to split the mass at our vertex with an independent $\mathrm{PD}(1/\alpha,1/\alpha-1)$.  A size-biased block, $B^*$, from this has law $\mathrm{Beta}(1-1/\alpha,2/\alpha-1)$.  So the relevant subtree is an independent stable tree, randomly rescaled to have mass $X_3: = (1-X_1-X_2) B^*$ and lengths rescaled by $X_3^{1-1/\alpha}$.  So the branch leading to our uniform vertex has length
\[
X_3 ^{1-1/\alpha} M^{(3)}
\]
where $M^{(3)}$ has $\mathrm{ML}(1-1/\alpha,1-1/\alpha)$ distribution, independently of everything else. Let $X_4 = (1-X_1-X_2) (1-B^*)$. Then $(X_1, X_2, X_3, X_4) \sim \mathrm{Dir}(1-1/\alpha, 1- 1/\alpha, 1-1/\alpha, 2/\alpha-1)$ and
\[
(X_1^{1-1/\alpha} M^{(1)}, X_2^{1-1/\alpha} M^{(2)}, X_3^{1-1/\alpha} M^{(3)})
\]
represents the lengths of the three edges present in the tree spanned by two uniform points and the root.  By Proposition~\ref{prop:lengthsmasses}, $L^{(\alpha,2)}$ does indeed have the claimed law.  The laws of $L^{(\alpha,p)}$ may be determined in a similar manner.
\end{proof}

\begin{rem} One aspect of this viewpoint remains mysterious to us: if we let $M^{(4)} \sim \mathrm{ML}(1-1/\alpha, 2/\alpha-1)$ then
\[
(X_1^{1-1/\alpha} M^{(1)}, X_2^{1-1/\alpha} M^{(2)}, X_3^{1-1/\alpha} M^{(3)}, X_4^{1-1/\alpha} M^{(4)})
\]
has the same distribution as the three lengths and the node weight after a single step of version (II) of our construction.  What (if anything) is the correct interpretation of the random variable $M^{(4)}$ in the tree? The $\mathrm{ML}(1-1/\alpha,2/\alpha-1)$ distribution does \emph{not} appear to be connected to the remaining parts of the fine spinal mass partition at the chosen vertex, which are distributed as a scaled copy of $\mathrm{PD}(1/\alpha,2/\alpha-1)$.  Is there a sensible way to associate a notion of ``length'' with the vertex?
\end{rem}

\begin{rem}
Proposition \ref{prop:otherproof} can equally be proved by considering the scaling limits of distances between pairs of points in Marchal's algorithm and using asymptotic results on generalized P\'olya urns. This gives an alternative explanation for the presence of generalized Mittag-Leffler distributions. 
\end{rem}

\begin{rem}
The self-similarity of the stable trees is not immediately obvious from our constructions.  However, we believe that it could be proved using the analogue of the arguments presented in Section 6 of \cite{ABBrGo} for the Brownian CRT.  We do not pursue this further here.
\end{rem}

\section*{Acknowledgments}
C.G.\ would like to thank the CEREMADE at Universit\'e Paris-Dauphine for having made her invited professor for the month of February 2014, during which time most of this work was done.

\bibliography{stable}

\end{document}